\documentclass[10pt,a4paper]{article}
\usepackage{setspace,framed,pifont,dingbat} 
\usepackage{amsmath,soul,color,amssymb,graphicx,rotating,colortbl,xcolor} 
\usepackage[amsmath,thmmarks,amsthm,framed]{ntheorem} 
\usepackage{typearea} 
\usepackage[T1]{fontenc} 
\usepackage{lmodern}
\usepackage[latin1]{inputenc} 
\usepackage{slashed}
\usepackage{mathtools}
\usepackage{mathrsfs}
\usepackage{titletoc,titlesec} 
\usepackage{microtype}
\usepackage{floatflt}
\usepackage[colorlinks=true,linkcolor=black, urlcolor=myblue2, citecolor = black]{hyperref} 
\usepackage{listings}
\usepackage{tikz}
\usepackage{subfigure}
\usepackage{graphicx}
\usepackage{authblk}

\titleformat{\chapter}[display] {\normalfont\Large\filcenter\normalfont}
{\titlerule[1pt]\vspace{1pt}%
\titlerule\vspace{1pc}%
\LARGE\MakeUppercase{\chaptertitlename} \thechapter} {1pc}{
\titlerule\vspace*{1pc}\Huge}
\titlespacing{\chapter}{0pt}{-3em}{3em}
\theoremstyle{break}

\theorembodyfont{\normalfont}

\newtheorem{thm}{Theorem}
\newtheorem{lem}{Lemma}

\newtheorem{defn}{Definition}
\newtheorem{rem}{Remark}

\newtheorem{prop}{Proposition}

\newcommand{\dx}{\,\text{d}}
\DeclareMathOperator*{\esssup}{ess\,sup}
\DeclareMathOperator*{\essinf}{ess\,inf}

\newcommand{\qa}{\mathbf{a}}
\newcommand{\qb}{\mathbf{b}}
\newcommand{\qe}{\mathbf{e}}
\newcommand{\qk}{\mathbf{k}}

\newcommand{\qm}{\mathbf{m}}
\newcommand{\qn}{\mathbf{n}}
\newcommand{\qp}{\mathbf{p}}

\newcommand{\qr}{\mathbf{r}}
\newcommand{\qs}{\mathbf{s}}
\newcommand{\qt}{\mathbf{t}}
\newcommand{\qu}{\mathbf{u}}
\newcommand{\qx}{\mathbf{x}}
\newcommand{\qy}{\mathbf{y}}

\newcommand{\qE}{\mathbf{E}} %Kernfunktion
\newcommand{\FT}{\mathcal{F}} %FOURIER
\newcommand{\ZT}{\mathcal{Z}} %ZAK
\newcommand{\QT}{\mathcal{Q}} %Two Sided QFT
\newcommand{\Gq}{\mathcal{G}_q} %Quaternion Gabor G

\newcommand{\QH}{\mathbb{H}}
\newcommand{\Sc}{\text{Sc}}
\newcommand{\VEC}{\text{Vec}}
\begin{document}
\begin{center}
\textbf{\Large Some results on the lattice parameters of quaternionic Gabor frames}
\end{center}
\begin{center}
S. Hartmann
\end{center}

\vspace{2em}
\begin{abstract}
Gabor frames play a vital role not only modern harmonic analysis but also in several fields of applied mathematics, for instances, detection of chirps, or image processing. In this work we present a non-trivial generalization of Gabor frames to the quaternionic case and give new density results. The key tool is the two-sided windowed quaternionic Fourier transform (WQFT). As in the complex case, we want to write the WQFT as an inner product between a quaternion-valued signal and shifts and modulates of a real-valued window function. We demonstrate a Heisenberg uncertainty principle and for the results regarding the density, we employ the quaternionic Zak transform to obtain necessary and sufficient conditions to ensure that a quaternionic Gabor system is a quaternionic Gabor frame. We conclude with a proof that the Gabor conjecture do not hold true in the quaternionic case.
\end{abstract}

\section{Introduction}
As a generalization of the real and complex Fourier transform (FT), the quaternionic Fourier transform (QFT) with two exponential kernels using two non-commutative complex variables, receives an increasing attention in the representation of signals \cite{Bahri-Quat_Gabor,Hitzer-Uncertainty,BAS,BULOWPHD,TAELL,Ell2014,BalianLowforWQFT,Hitzer-QFT}. There is some degree of freedom as where one can applies the exponential kernels, see for example \cite{Bahri-Quat_Gabor, TAELL}. The right-sided version has been thoroughly studied by M. Bahri et a.l. in \cite{Bahri-Quat_Gabor, Bahrietal, Hitzer-Uncertainty}, among others. In this paper we are going to use the two-sided (sandwiched) version, which leads to more symmetric properties than the one-sided version \cite{Bahri-Quat_Gabor, Hitzer-Uncertainty, BalianLowforWQFT, Hitzer-QFT}. Nevertheless, as in the case of the FT, the QFT cannot capture the features of instationnary signals. To overcome this problem  
%lacks the property of representing and computing local information of quaternionic signals. Therefore,
 we multiply the signal with translations and modulations of a real-valued window function and then apply the QFT. This procedure yields a joint representation of locality and frequency. Hence, the window function is also of importance in the synthesis of the signal. In applications it is required to reconstruct the signal from a discrete set of data in a given lattice. The properties of this lattice are crucial for the reconstruction. In the non-quaternionic case three different situations can occur, which depend on the lattice parameters $\alpha$ and $\beta$. If the product of these two parameters is greater than $1$, a reconstruction is not possible, whereas it is possible for certain functions, e.g. the Gaussian function, if $\alpha \beta < 1$. The third case, $\alpha \beta = 1$, also known as \emph{critical density}, is interesting since it allows for the existence of %is associated to 
 orthogonal basis. This leads to the study of the synthesis of the signal under certain circumstances.  This situation is still under investigation, see \cite{Borichev, HEIL, Janssentight,  ClaasroomProof}.\\
In the quaternionic case, the influence of the density of the lattice in the reconstruction of the signal is not yet fully studied. Based on that, we focus our attention on the critical density case $\alpha \beta = 1$. Since the Gaussian function also minimizes the quaternionic version of the Heisenberg uncertainty principle, we are especially interested in the Gaussian window at the critical density. Our final result is a partial generalization of the Lyubarskii and Seip-Wallstens Theorem \cite{Seip1, Seip2}.\\
The paper is organized as follows: In Section 2 we briefly sketch some results on quaternionic algebra and Wiener-Amalgam spaces. In Section 3 we summarize (without proofs) the relevant material on WQFT and develop the theory of quaternionic Gabor frames. As the Zak transform will be our main tool for obtaining density results, we introduce in this section the quaternionic Zak transform and establish its main properties (unitarity, periodicity, inversion formula, etc.).  In Section 4 we provide the connections between the Zak transform of the window function and the frame bounds (related to the parameters $\alpha$ and $\beta$), together with the conditions under which the quaternionic Gabor system constitutes a quaternionic Gabor frame. % The significance of the results of this paper is not only of abstract nature, but also useful for application in image processing.   

%%%%%%%%%%%%%%%%%%%%%%%%%%%%%%%%%%%%%%%%%%%%%%%%%%%%%%%%%%%%%%%%%%%%%%%%%%%%%%%%%%%%%%%%%%%%% PRELIMINARIES 
\section{Preliminaries}
\subsection{Quaternions}
The \emph{quaternion algebra} $\QH$ is an extension of the complex numbers into four dimensions and is given by
\[ \QH = \{ q \,|\, q = q_0 + q_1 i + q_2 j + q_3 k, \text{\  with\  } q_0, q_1, q_2,q_3 \in \mathbb{R}\}, \]
where the elements $i,j,k$ satisfy
\[ ij = -ji = k \qquad i^2 = j^2 = k^2 = ijk = -1. \]

We can write a quaternion $q$ also as a sum of a \emph{scalar} $\Sc[q] = q_0 \in \mathbb{R}$ and a three-dimensional vector $\VEC[q] = q_1 i + q_2 j + q_3 k$, $\VEC[q]$ is often called \emph{pure quaternion},
\[ q  = \Sc[q] + \VEC[q]. \]
The conjugate is an automorphism on $\mathbb{H}$, given by $q \mapsto \overline{q} = \Sc[q] - \VEC[q]$. Moreover
\[ q \overline{q} = |q|^2. \]
Although quaternions are noncommutative, in general, a \emph{cyclic multiplication} property exists. This feature is going to be a useful tool for the rest of this work.
\begin{lem}[Cyclic multiplication \cite{Sprssig}]
For all $q,r,s \in \QH$ holds
\begin{equation}\label{eq:SC} \Sc[q r s] = \Sc[r s q] = \Sc[s q r]. \end{equation}
\end{lem}

\subsection{Related function spaces}
In the following we use the space $L^2(\mathbb{R}^2,\mathbb{H})$, an immediate generalization of the Hilbert space of all square-integrable functions, which consists of all quaternion-valued functions $f: \mathbb{R}^2 \to \mathbb{H}$ with finite norm
\[ \|f\|_2 = \left(\int_{\mathbb{R}^2} |f(\qx)|^2 \dx^2 \qx \right)^{1/2} < \infty \]
where $\dx^2 \qx = \dx x_1 \dx x_2$ represents the usual Lebesgue measure in $\mathbb{R}^2$. The $L^2$-norm is induced by the symmetric real scalar product
\begin{equation}\label{eq:Sc}
\langle f, g \rangle =  \Sc \int_{\mathbb{R}^2} f(\qx) \overline{g(\qx)} \dx^2 \qx = \frac{1}{2}\int_{\mathbb{R}^2} (g(\qx) \overline{f(\qx)} + f(\qx) \overline{g(\qx)})\dx^2 \qx \end{equation}
which makes the space a real linear space. 
\begin{rem}[\cite{Hitzer-QFT}]
It is also possible to define a quaternion-valued inner product on $L^2(\mathbb{R}^2,\mathbb{H})$ with
\[ (f,g) = \int_{\mathbb{R}^2} f(\qx) \overline{g(\qx)} \dx^2 \qx \]
in which case one obtains a \emph{left Hilbert module}.
We remark that for a quaternion-valued inner product one has the following scalar-product rules
\[ (\lambda f, g) = \lambda (f,g), \qquad  (f, g \lambda) = (f, g) \overline{\lambda}, \qquad \lambda \in \mathbb{H} . \]
Both inner products lead to the same norm.  
\end{rem}

The two-dimensional cube $[0,\alpha]^2$ will be denoted by $Q_{\alpha}$ and the special case of $\alpha = 1$ simply by $Q$.
The Wiener-Amalgam space is a suitable space from where to choose the window functions.
\begin{defn}[Wiener-Amalgam space \cite{Grchenig2001}]
A function $g \in L^\infty(\mathbb{R}^2, \mathbb{R})$ belongs to the Wiener-Amalgam space $W^\alpha = W^\alpha(\mathbb{R}^2)$ if
\begin{equation} \|g\|_{W^\alpha} = %\sum_{\qn \in \mathbb{Z}^2} \|g T_{\alpha \qn} \mathcal{X}_{Q_\alpha}\|_{\infty} =
 \sum_{\qn \in \mathbb{Z}^2} \esssup_{\qx \in \mathbb{R}^2} |g(\qx) \mathcal{X}_{Q_\alpha}(\qx - \alpha \qn)|  \label{defn:wiener}
 \end{equation}
is finite.
\end{defn}
The subspace of the continuous functions is denoted by $W_0^\alpha = W^\alpha \cap C$.\\ 
%The following lemma will be required in future results. A proof can be found in \cite[p. 105]{Grchenig2001}.
%\begin{lem}[\cite{Grchenig2001}]\label{HelpLemma}
%If $g \in W(\mathbb{R}^2)$ and $\eta > 0$, then
%\[ \esssup_{\qx \in \mathbb{R}^2} \sum_{\qn \in \mathbb{Z}^2} |g(\qx - \eta \qn)| \leq \left(\frac{1}{\eta} + 1\right)^2 \|g\|_W.\]
%\end{lem}
%\begin{proof}
%Since any interval of length $1$ contains at most $\left[ \frac{1}{\eta} \right] + 1$ distinct points of minimum distance $\eta$ between each other, the cube $l + Q = \prod_{j = 1}^2 [l_j, l_j + 1]$ contains at most $(\left[ \eta^{-1} \right] + 1)^2$ points of the form $\qx + \eta \qn$ with $\qn \in \mathbb{Z}^2$, independently of $\qx \in \mathbb{R}^2$. Therefore
%\begin{align*}
%\sum_{\qn \in \mathbb{Z}^2} |g(\qx - \eta \qn)| &\leq \sum_{\ql \in \mathbb{Z}^2} \left( \left[ \frac{1}{\eta} \right] + 1 \right)^2 \sup_{\{\qn\,:\,\qx - \eta \qn \, \in \, \ql + Q\}} |g(\qx - \eta \qn)|\\
%&\leq \left( \frac{1}{\eta} + 1\right)^2 \sum_{\ql \in \mathbb{Z}^2} \|g \cdot T_\ql \mathcal{X}_{Q}\|_\infty\\
%&= \left(\frac{1}{\eta} + 1\right)^2 \|g\|_W.
%\end{align*}
%\end{proof}

\section{Quaternionic Gabor frames $\Gq(g, \alpha, \beta)$}
%%%%%%%%%%%%%%%%%%%%%%%%%%%%%%%%%%%%%%%%%%%%%%%%%%%%%%%%%%%%%%%%%%%%%%%%%%%%%%%%%%%%%%%%%%%%%%%%%%%%%%%%% QUATERNIONIC GABOR TRANSFORM
In this section we extend the complex Fourier transform to the two-sided quaternionic setting \cite{BULOWPHD, TAELL, BalianLowforWQFT}. We show some properties of the two-sided quaternionic Fourier transform. With the same motivation as in the short-time real-valued Fourier transform we have to multiply the signal with a window function $g$ to obtain local information of the frequency. This idea leads to quaternionic Gabor frames and related operators, which will be discussed in the second part. The definition of the Zak transform in the third part arises naturally and will be our main tool for the density results.  
\subsection{The quaternionic Gabor transfom}
\begin{defn}[Two sided quaternionic Fourier transform (QFT) \cite{BalianLowforWQFT}]
The \emph{two sided quaternionic Fourier transform} of $f \in L^2(\mathbb{R}^2,\mathbb{H})$ is the function $\FT_q(f)\,:\, \mathbb{R}^2 \to \mathbb{H}$ defined by
\[ \omega = (\omega_1, \omega_2) \mapsto \FT_q(f)(\omega) = \widehat{f}(\omega) = \int_{\mathbb{R}^2} \exp(-2 \pi i x_1 \omega_1) f(\qx) \exp(-2 \pi j x_2 \omega_2) \dx^2 \qx. \]
\end{defn}

\begin{rem}[\cite{BalianLowforWQFT}]
For the reconstruction of $f$ we obtain
\begin{equation}\label{eq:REC} \qx = (x_1, x_1) \mapsto f(\qx) = \FT_q^{-1}(\FT_q f)(\qx) = \int_{\mathbb{R}^2} \exp(2 \pi i x_1 \omega_1) \widehat{f}(\omega) \exp(2 \pi j x_2 \omega_2) \dx^2 \omega. \end{equation}
\end{rem}

There is also a Heisenberg uncertainty principle in the QFT case. Moreover, the two-dimensional Gaussian function minimizes the uncertainty. % Since the proof is similar to \cite{Hitzer-Uncertainty}, we skip them. 
\begin{defn}[\cite{Hitzer-Uncertainty}]
Let $f \in L^2(\mathbb{R}^2, \mathbb{H})$ be such that $x_k f \in L^2(\mathbb{R}^2,\mathbb{H})$ and let $\widehat{f} \in L^2(\mathbb{R}^2, \mathbb{H})$ be its QFT such that $\omega_k \widehat{f} \in L^2(\mathbb{R}^2,\mathbb{H})$. The \emph{spatial uncertainty} $\Delta x_k$ is defined as
\[ \Delta x_k = \sqrt{\frac{1}{\|f\|^2_2} \int_{\mathbb{R}^2} |f(\qx)|^2 x_k^2 \dx^2 \qx} \]
and the \emph{spectral uncertainty} $\Delta \omega_k$ is defined as
\[ \Delta \omega_k = \sqrt{\frac{1}{\|\widehat{f}\|^2_2} \int_{\mathbb{R}^2} |\widehat{f}(\omega)|^2 \omega_k^2 \dx^2 \omega}, \quad \text{for\ } k = 1,2 \]
\end{defn}

\begin{thm}[Heisenberg uncertainty principle in QFT case]
Let $f \in L^2(\mathbb{R}^2,\mathbb{H})$ satisfy to $(1 + |x_k|)f(\qx) \in L^2(\mathbb{R}^2,\mathbb{H})$ and $\frac{\partial}{\partial x_k} f(\qx) \in L^2(\mathbb{R}^2,\mathbb{H})$. Then we have 
\[ \Delta x_k \Delta \omega_k \geq \frac{1}{4 \pi} \quad \text{for\ } k = 1,2. \]
Also, equality holds if and only if $f(\qx)$ is a Gaussian function.
\end{thm}
Since the proof is similar to \cite{Hitzer-Uncertainty}, we omit it.

\begin{defn}[Windowed quaternionic Fourier transform \cite{BalianLowforWQFT}]
The windowed quaternionic Fourier transform (WQFT) of $f \in L^2(\mathbb{R}^2, \mathbb{H})$ with respect to a non-zero window function $g \in L^2(\mathbb{R}^2, \mathbb{R})$ is defined as
\[ \QT_g f(\qb, \omega) = \int_{\mathbb{R}^2} \exp(-2 \pi i x_1 \omega_1) f(\qx) g(\qx-\qb) \exp(-2 \pi j x_2 \omega_2) \dx^2 \qx, \qquad (\qb, \omega) \in \mathbb{R}^2 \times \mathbb{R}^2.
\]
\end{defn}

In order to write $\QT_g f(\qb, \omega)$ as an inner product of $f$ with translates by $\qb$ and modulations by $\omega$ of $g$, we use the concept of carriers, introduced by Shapiro et al. \cite{Shapiro}.

\begin{defn}[Carrier \cite{Shapiro}]
For two quaternions $p,q \in \mathbb{H}$ we define the \emph{right} $C_r$ and \emph{left} $C_l$ \emph{carrier operators} as
\[ C_r(p)q = q p \qquad \text{and} \qquad q C_l(p) = p q. \] 
\end{defn}

\begin{lem}[Properties of the carrier]
The carriers have the following properties with $p \in \mathbb{H}$\\
(a)\ $\overline{C_r(p)} = C_l(\overline{p})$ \  and \  $\overline{C_l(p)} = C_r(\overline{p})$,\\
(b)\ $C_r(p)1 =  1C_l(p) = p$.
%(c)\ $C_r(p) 1 C_l(q) = q p$ \  and \  $1 C_l(p) C_r(q) 1 = p q$,\\
%(d)\ $C_r(\exp(m i)) C_r(\exp(n i)) = \exp((m+n)i).$
\end{lem}

Now, we are equipped with the necessary tools for the generalization of translations and modulations to our setting.
%\begin{defn}
%We formally define the inner product of two functions $f, g$ as
%\[ (f|g) = f(\qx) g(\qx) \]
%\end{defn}
%\begin{rem}
%\textcolor{blue}{Maybe put some properties}
%\end{rem}

\begin{defn}
For $\qb, \omega \in \mathbb{R}^2$ we define the following operators
\[ T_{\qb} g(\qx) = g(\qx - \qb) \qquad \text{translation by $\qb \in \mathbb{R}^2$} \] and
\[ M_\omega g(\qx)  = \exp(2 \pi j \omega_2 x_2) g(\qx)C_r(\exp(2 \pi i \omega_1 x_1)) \qquad \text{modulation by $\omega \in \mathbb{R}^2$}. \]
\end{defn}
With the help of those two operators we rewrite the WQFT as
\[ \QT_g f(\qb, \omega) = (f, M_\omega T_\qb g). \]

Since the translation and modulation only act on the window function $g$, it is quite interesting to point out properties between translation and modulation.  
\begin{rem}\label{rem:COMMURELA}
Those two operators satisfy the following \emph{commutation relation}
\begin{equation}\label{eq:commuq}
T_\qb M_\omega g(\qx) = \exp(2 \pi j b_2 \omega_2) M_\omega T_\qb g(\qx) \exp(2 \pi i b_1 \omega_1). 
\end{equation}
As a result, translation and modulation commute if and only if $b_1 \omega_1$ and $b_2 \omega_2 \in \mathbb{Z}$. We also obtain the following relation which displays the interplay between the two operators and the QFT:
\[
\widehat{T_\qb g}(\omega) = \overline{M_{\qb}} \widehat{g}(\omega) \qquad \text{and} \qquad T_\omega \widehat{g}(\xi) = \widehat{\overline{M_{-\omega}}g}(\xi).
\]

Also, based on these operators we can express the norm (\ref{defn:wiener}) of the Wiener-Amalgam space as 
\[ \|g\|_{W^\alpha} = \sum_{\qn \in \mathbb{Z}^2} \|g T_{\alpha \qn} \mathcal{X}_{Q_\alpha}\|_{\infty}. \]%= \sum_{\qn \in \mathbb{Z}^2} \esssup_{\qx \in \mathbb{R}^2} |g(\qx) \mathcal{X}_{Q_\alpha}(\qx - \alpha \qn)| \] 
\end{rem}

We now look into some properties of the WQFT.
\begin{thm}
Let $g \in L^2(\mathbb{R}^2,\mathbb{R})$ be a non-zero window function and $f\in L^2(\mathbb{R}^2, \mathbb{H})$. Then, we have
\[ \QT_g(T_{\qx_0}f)(\qb,\omega) = \exp(-2 \pi i x_{01}\omega_1) \QT_g f( \qb - \qx_0,\omega) \exp(- 2 \pi j x_{02} \omega_2) \]
and 
\[ \QT_g(\overline{M_{-\omega_0}} f)(\qb,\omega) = \QT_g f(\qb, \omega - \omega_0). \]
%where the translation operator is $T_{\qx_0}f(\qx) = f(\qx - \qx_0)$ and the modulation operator is $\overline{M_{-\omega_0}} f = \exp(2 \pi i x_1 \omega_{01}) f(\qx) \exp(2 \pi j x_2 \omega_{02})$.
\end{thm}

\begin{prop}[Reconstruction formula]
Let $g \in L^2(\mathbb{R}^2, \mathbb{R})$ be a non-zero window function. Then, every $f \in L^2(\mathbb{R}^2,\mathbb{H})$ can be fully reconstructed by
\[
f(\qx) = \frac{1}{\|g\|^2_2}\int_{\mathbb{R}^2}\int_{\mathbb{R}^2} \exp(2 \pi i x_1 \omega_1) \QT_g f(\qb, \omega)g(\qx - \qb) \exp(2 \pi j x_2 \omega_2) \dx ^2 \omega \dx^2 \qb.
\]
\end{prop}

\begin{thm}[Orthogonality relation]
Let $g \in L^2(\mathbb{R}^2,\mathbb{R})$ be a non-zero window function and $f, h \in L^2(\mathbb{R}^2,\mathbb{H})$. Then,  $\QT_g f, \QT_g h \in L^2(\mathbb{R}^2,\mathbb{H})$ and
\[ \langle \QT_g f, \QT_g h \rangle = \|g\|^2_2 \langle f, h \rangle.  \]
\end{thm}
For the proofs we refer to \cite{BalianLowforWQFT}.

%%%%%%%%%%%%%%%%%%%%%%%%%%%%%%%%%%%%%%%%%%%%%%%%%%%%%%%%%%%%%%%%%%%%%%%%%%%%%%%%%%%%%%%%%%%%%%%%%%%%%% ANALYSIS AND SYNTHESIS
\subsection{The analysis and synthesis operators}
In applications we have to replace integrals over $\mathbb{R}^2 \times \mathbb{R}^2$ by sums over a four-dimensional lattice in space and frequency. Hence, the properties of the lattice are crucial for the decomposition and reconstruction of the signal. The question of synthesis of the signal is the main key of this paper, as it influences both its decomposition and reconstruction. For obtaining further results with respect to the synthesis we will require the following definitions and theorems. In this, we follow the notations of K. Gröchenig \cite{Grchenig2001}.
 
\begin{defn}[Quaternionic Gabor frame]
Given a non-zero $g \in L^2(\mathbb{R}^2,\mathbb{R})$ and lattice parameters $\alpha, \beta > 0$, then the set of space-frequency shifts 
\[ \Gq(g,\alpha,\beta) = \{ M_{\beta \qn} T_{\alpha \qm} g, \text{\  with \ } \qm, \qn \in \mathbb{Z}^2 \} \]
is called a \emph{quaternionic Gabor system}. The set $\Gq(g,\alpha,\beta)$ is a frame for $L^2(\mathbb{R}^2,\mathbb{H})$ if there exist real constants $A,B > 0$ such that
\[ A \|f\|^2 \leq \sum_{\qm,\qn \in \mathbb{Z}^2} |\langle f, M_{\beta \qn} T_{\alpha \qm} g\rangle|^2 \leq B \|f\|^2 \qquad \forall f \in L^2(\mathbb{R}^2,\mathbb{H}). \]
\end{defn}

\begin{defn}
For $\Gq(g, \alpha, \beta) = \{M_{\beta \qn} T_{\alpha \qm}g\,:\, \qm, \qn \in \mathbb{Z}^2\} \subseteq L^2(\mathbb{R}^2,\mathbb{H})$ we define the \emph{coefficient operator}, or \emph{analysis operator}, $C_q$ by 
\begin{equation*} \label{eq:Cq}
C_q f = \{\langle f, M_{\beta \qn} T_{\alpha \qm} g \rangle\,:\, \qm, \qn \in \mathbb{Z}^2\}. 
\end{equation*}
Conversely, given a real-valued sequence $c = \{c_{\qm,\qn}\,:\,\qm,\qn \in \mathbb{Z}^2 \}$ we define the \emph{reconstruction operator}, or \emph{synthesis operator}, $D_q$ by
\begin{equation*} \label{eq:Dq}
D_q c = \sum_{\qm, \qn \in \mathbb{Z}^2} \exp(2 \pi i \beta n_1) c_{\qm, \qn} g(\qx - \alpha \qm) \exp(2 \pi j \beta n_2) = \sum_{\qm,\qn \in \mathbb{Z}^2} c_{\qm, \qn}  \overline{M_{-\beta \qn}} T_{\alpha \qm} g %\in L^2(\mathbb{R}^2,\mathbb{H}).
\end{equation*}
In consequence, the associated \emph{frame operator} $S_q$ on $L^2(\mathbb{R}^2,\mathbb{H})$ has the following form
\begin{equation*} \label{eq:Sq}
 S_q f = D_q C_q f = \sum_{\qm,\qn \in \mathbb{Z}^2} \langle f, M_{\beta \qn} T_{\alpha \qm} g \rangle \overline{M_{-\beta \qn}} T_{\alpha \qm} g. 
\end{equation*}
\end{defn}

%%%%%%%%%%%%%%%%%%%%%%%%%%%%%%%%%%%%%%%%%%%%%%%%%%%%%%%%%%%%%%%%%%%%%%%%%%%%%%%%%%%%%%%%%%%%%%%%%%%%%%%%%%%%%%%%%% ZAK
\subsection{The quaternionic Zak transform }
The quaternionic Zak transform is a natural extension of the Zak transform into the two-sided quaternionic setting. Therefore, most of its properties are preserved in this setting. In harmonic analysis the Zak transform is also known as \emph{Weil-Brezin map}.  
\begin{defn}[Quaternionic Zak transform]
The \emph{quaternionic Zak transform} of a quaternion-valued function $f \in L^2(\mathbb{R}^2,\mathbb{H})$ is defined as
\[ \ZT^\alpha_qf(\qx, \omega) = \sum_{\qm \in \mathbb{Z}^2} \exp(2 \pi j \alpha m_2 \omega_2) f(\qx - \alpha \qm) C_r(\exp(2 \pi i \alpha m_1 \omega_1)), \qquad \qx, \omega \in \mathbb{R}^2. \]
\end{defn}

%\begin{thm}[\cite{Zakquat}]
%The normalized Zak transform $\alpha \ZT^\alpha_q$ is an unitary map from $L^2(\mathbb{R}^2,\mathbb{H})$ onto $L^2(Q^2_\alpha \times Q^2_{1/\alpha}, \mathbb{H})$.
%\end{thm}
The following theorem shows that the quaternionic Zak transform is a unitary operator, mapping $L^2(\mathbb{R}^2, \mathbb{H})$ onto $L^2(Q_{\alpha} \times Q_{1/\alpha}, \mathbb{H})$, up to a constant.
\begin{thm}\label{thm:41}
For $\varphi, \psi \in L^2(\mathbb{R}^2,\mathbb{H})$, it holds
\[ \langle \ZT^\alpha_q \varphi, \ZT^\alpha_q \psi \rangle = \alpha^{-2}\langle \varphi, \psi \rangle. \]
In particular,
\[ \|\ZT^{\alpha}_q \varphi\|^2_{L^2(Q_\alpha \times Q_{1/\alpha})} = \int_{Q_\alpha} \int_{Q_{1/\alpha}} |\ZT_q^{\alpha} \varphi(\qx, \omega)|^2 \dx^2 \qx \dx^2 \omega = \alpha^{-2} \|\varphi\|^2_2. \] 
\end{thm}

\begin{prop}[Quasiperiodicity conditions]
The transform $\ZT^\alpha_q$ fulfills
\[ \ZT^\alpha_q f(\qx,\omega + \frac{\mathbf{1}}{\alpha}) = \ZT^\alpha_q f(\qx,\omega), \qquad \ZT_q^\alpha f(\qx + \mathbf{1}\alpha, \omega) = \exp(2 \pi j \alpha \omega_2) \ZT_q^\alpha f(\qx, \omega) \exp(2 \pi i \alpha \omega_1)). \]
\end{prop}
Thus, $\ZT_q^\alpha$ is periodic and it is sufficient to know its values in the cube $Q_\alpha \times Q_{1/\alpha}$. In the next Lemma we proof an inversion formula for the signal $f$ by means of its Zak transform. Combining these results, we conclude with the reconstruction of the signal $f$ based on the values of its Zak transform on the cube $Q_{1/\alpha}$.

\begin{thm}[Inversion formula \cite{Grchenig2001}]
If $f \in L^2(\mathbb{R}^2,\mathbb{H})$ then, 
\[ f(\qx) = \alpha^{2} \int_{Q_{1/\alpha}} \ZT^\alpha_q f(\qx, \omega) \dx^2 \omega, \qquad \qx \in \mathbb{R}^2. \]
\end{thm}
This allows us to recapture the whole function $f$ solely by the values of $\ZT_q^\alpha f$ on $Q_{1/\alpha}$. 

\begin{lem}\label{lem:lem37}
We have for the Zak transform of %a quaternionic space-frequency shift 
$M_{\frac{\qk}{\alpha}} T_{\alpha \qn} g$
\begin{align*}
&\ZT_q^{\alpha} (M_{\frac{\qk}{\alpha}} T_{\alpha \qn} g)(\qx, \omega) \\
&= \exp(-2 \pi j \alpha n_2 \omega_2) \exp(2 \pi j k_2 \frac{x_2}{\alpha}) \ZT_q^{\alpha} g(\qx, \omega)\exp(2 \pi i k_1 \frac{x_1}{\alpha})\exp(-2 \pi i \alpha n_1 \omega_1).
\end{align*}
\end{lem}
The proof is an easy calculation and, therefore, it will be omitted. Nevertheless, the result is useful for the next lemma. At this point we remark that this corresponds to assume $\beta = \frac{1}{\alpha}$, that is the critical density case.
 
\begin{lem}\label{lem:47}
If $f \in L^2(\mathbb{R}^2,\mathbb{H})$ and $g \in L^2(\mathbb{R}^2, \mathbb{R})$, then we have
\[ \sum_{\qk, \qn \in \mathbb{Z}^2} |\langle f, M_{\frac{\qk}{\alpha}} T_{\alpha \qn} g \rangle|^2 = \alpha^4 \|\ZT_q^{\alpha} f \overline{\ZT_q^{\alpha} g} \|^2_{L^2(Q_\alpha \times Q_{1/\alpha})}.  \]
\end{lem}
\begin{proof}
By Theorem \ref{thm:41} and Lemma \ref{lem:lem37}
\begin{align*}
&\sum_{\qk, \qn \in \mathbb{Z}^2} |\langle f, M_{\frac{\qk}{\alpha}} T_{\alpha \qn} g \rangle|^2\\
& = \sum_{\qk, \qn \in \mathbb{Z}^2} |\langle \alpha \ZT_q^{\alpha} f, \alpha \ZT_q^{\alpha} M_{\frac{\qk}{\alpha}} T_{\alpha \qn} g \rangle|^2\\
& = \alpha^4 \sum_{\qk, \qn \in \mathbb{Z}^2} | \Sc \int_{Q_{\alpha}} \int_{Q_{1/\alpha}} \ZT_q^{\alpha} f \  \overline{\ZT_q^{\alpha} M_{\frac{\qk}{\alpha}} T_{\alpha \qn}} \dx^2 \qx \dx^2 \omega|^2\\
& = \alpha^4 \sum_{\qk, \qn \in \mathbb{Z}^2} | \Sc \int_{Q_{\alpha}} \int_{Q_{1/\alpha}} \ZT_q^\alpha f C_l(\exp(2 \pi i ( \alpha n_1 \omega_1 - \frac{k_1}{\alpha} x_1)))\\
&\qquad  \overline{\ZT_q^{\alpha}g} \exp(-2 \pi j \frac{k_2}{\alpha} x_2)\exp(2 \pi j \alpha n_2 \omega_2) \dx^2 \qx \dx^2 \omega|^2\\
&= \alpha^2 \sum_{\qk, \qn \in \mathbb{Z}^2} | \Sc \int_{Q_{\alpha}} \int_{Q_{1/\alpha}} \exp(2 \pi i \alpha n_1 \omega_1) \exp(-2 \pi i \frac{k_1}{\alpha} x_1) \ZT_q^{\alpha} f \overline{\ZT_q^{\alpha} g} \\
&\qquad  \exp(-2 \pi j \frac{k_2}{\alpha} x_2) \exp(2 \pi j \alpha n_2 \omega_2) \dx^2 \qx \dx^2 \omega|^2
\end{align*}
\begin{align*}
&= \alpha^{4} \sum_{\qk, \qn \in \mathbb{Z}^2} |\Sc \int_{Q_{\alpha}} \int_{Q_{1/\alpha}}  \ZT^\alpha_q f \overline{\ZT^\alpha_q g} \exp(-2 \pi j \frac{k_2}{\alpha} x_2) \exp(2 \pi j \alpha n_2 \omega_2)\\
&\qquad \exp(2 \pi i \alpha n_1 \omega_1) \exp(-2 \pi i \frac{k_1}{\alpha} x_1) \dx^2 \qx \dx^2 \omega|^2\\
&= \alpha^{4} \sum_{\qk, \qn \in \mathbb{Z}^2} |\langle \ZT_q^{\alpha} f \overline{\ZT^\alpha_q g}, \exp(-2 \pi j \frac{k_2}{\alpha} x_2) \exp(2 \pi j \alpha n_2 \omega_2) \exp(2 \pi i \alpha n_1 \omega_1) \exp(-2 \pi i \frac{k_1}{\alpha} x_1) \rangle|^2 \\
&= \alpha^4 \|\ZT_q^\alpha f \overline{\ZT_q^{\alpha} g}\|^2_{L^2(Q_\alpha \times Q_{1/\alpha})}
\end{align*}
since $\exp(-2 \pi j \frac{k_2}{\alpha} x_2) \exp(2 \pi j \alpha n_2 \omega_2) \exp(2 \pi i \alpha n_1 \omega_1) \exp(-2 \pi i \frac{k_1}{\alpha}x_1)$ is an orthonormal basis for $L^2(Q_{\alpha} \times Q_{1/\alpha}, \mathbb{H})$.
\end{proof}

%%%%%%%%%%%%%%%%%%%%%%%%%%%%%%%%%%%%%%%%%%%%%%%%%%%%%%%%%%%%%%%%%%%%%%%%%%%%%%%%%%%%%%%%% SOME DENSITY RESULTS FOR 
\section{Some density results for $\Gq(g, \alpha, \beta)$}
In this section we use the quaternionic Zak transform to show under which conditions a quaternionic Gabor system generates a quaternionic Gabor frame.
\begin{lem}\label{lem:GqFrame} 
For $g \in L^2(\mathbb{R}^2, \mathbb{R}),$ and $\alpha >0$ we have that  $\Gq(g, \alpha, \frac{1}{\alpha})$ is a frame for $L^2(\mathbb{R}^2,\mathbb{H})$ if and only if there exist $0 < a \leq b < \infty$ such that
\[ 0 < a \leq |\ZT_q^{\alpha} g(x, \omega)|^2 \leq b < \infty \qquad \text{almost everywhere in } Q_{\alpha} \times Q_{1/\alpha}.\]
Moreover, the optimal frame  bounds are then given by
\begin{align*}
A_{\text{opt}} &= \alpha^2 \essinf_{(\qx, \omega) \in Q_\alpha \times Q_{1/\alpha}} |\ZT_q^\alpha g(\qx, \omega)|^2,\\
B_{\text{opt}} &= \alpha^2 \esssup_{(\qx, \omega) \in Q_\alpha \times Q_{1/\alpha}} |\ZT_q^\alpha g(\qx, \omega)|^2.
\end{align*}
\end{lem}

\begin{proof} For the first statement, assume that $\Gq(g, \alpha, \frac{1}{\alpha})$ is a frame for $L^2(\mathbb{R}^2,\mathbb{H}),$ that is to say, 
\[ A \| f\|_2^2 \leq \sum_{\qk, \qn \in \mathbb{Z}^2} | \langle  f, M_{\frac{\qk}{\alpha}} T_{\qn \alpha} g\rangle   | \leq B \| f\|_2^2.\]
By Lemma \ref{lem:47} we get that, for all $F = \ZT_q^{\alpha} f \in L^2(Q_{\alpha} \times Q_{1/\alpha},\mathbb{H})$, it holds
\[ A \alpha^2 \|F\|^2_{L^2(Q_{\alpha} \times Q_{1/\alpha})} \leq \alpha^4 \|F \overline{\ZT_q^{\alpha} g}\|_{L^2(Q_{\alpha} \times Q_{1/\alpha})}^2 \leq B \alpha^2 \|F\|^2_{L^2(Q_{\alpha} \times Q_{1/\alpha})}, \]
or
\[ A \|F\|^2_{L^2(Q_{\alpha} \times Q_{1/\alpha})} \leq \alpha^2 \|F \overline{\ZT_q^{\alpha} g}\|_{L^2(Q_{\alpha} \times Q_{1/\alpha})}^2 \leq B \|F\|^2_{L^2(Q_{\alpha} \times Q_{1/\alpha})}, \]
almost everywhere in $Q_{\alpha} \times Q_{1/\alpha}$. This implies $a \leq |\ZT_q^{\alpha}g(\qx, \omega)|^2 \leq b$ almost everywhere in $Q_{\alpha} \times Q_{1/\alpha},$ where $a=A$ and $b= B $. 
On the other hand, if $a \leq \alpha^2 |\ZT_q^{\alpha} g(\qx, \omega)|^2 \leq b$ almost everywhere in $Q_{\alpha} \times Q_{1/\alpha},$ then we obtain that
\[ a \| \ZT_q^{\alpha} f \|^2_{L^2(Q_{\alpha} \times Q_{1/\alpha})} \leq \alpha^2 \| \ZT_q^{\alpha} f \overline{\ZT_q^{\alpha} g}\|_{L^2(Q_{\alpha} \times Q_{1/\alpha})}^2 \leq b \|\ZT_q^{\alpha} f\|^2_{L^2(Q_{\alpha} \times Q_{1/\alpha})} \] and, by Lemma \ref{lem:47} and Theorem \ref{thm:41}, we conclude
\[ a \|f\|^2_2 \leq \alpha^4 \|\ZT_q^{\alpha} f \overline{\ZT_q^{\alpha} g} \|^2_{L^2(Q_{\alpha} \times Q_{1/\alpha})} = \sum_{\qk, \qn \in \mathbb{Z}^2} |\langle f, M_{\frac{\qk}{\alpha}} T_{\alpha \qn} g \rangle|^2 \leq b \|f\|^2_2. \] 
Thus, the proof of the first statement is complete.

For the optimal frame bounds, and from 
\[ A \| f\|_2^2 \leq \sum_{\qk, \qn \in \mathbb{Z}^2} | \langle  f, M_{\frac{\qk}{\alpha}} T_{\qn \alpha} g\rangle   | \leq B \| f\|_2^2, \]
we obtain 
\[ A\alpha^2  \|F\|^2_{L^2(Q_{\alpha} \times Q_{1/\alpha})} \leq \alpha^4 \|F \overline{\ZT_q^{\alpha} g}\|^2_{L^2(Q_{\alpha} \times Q_{1/\alpha})} \leq \alpha^2 B \|F\|^2_{L^2(Q_{\alpha} \times Q_{1/\alpha})}\]
almost everywhere in $Q_{\alpha} \times Q_{1/\alpha}.$
From the first inequality, we have that
\[A \int_{Q_{\alpha}} \int_{Q_{1/\alpha}} |F(\qx, \omega) |^2 \dx^2 \qx \dx^2 \omega \leq   \int_{Q_{\alpha} \times Q_{1/\alpha}} |F(\qx, \omega) |^2 \alpha^2 |\overline{\ZT_q^{\alpha} g(\qx, \omega)} |^2 \dx^2 \omega \] which implies 
\[ \int_{Q_{\alpha}} \int_{Q_{1/\alpha}} |F(\qx, \omega) |^2 \left(  \alpha^2 |  \overline{\ZT_q^{\alpha} g(\qx, \omega)} |^2  -A \right) \dx^2 \omega \dx^2 \qx, \]
almost everywhere in $Q_{\alpha} \times Q_{1/\alpha}.$ Hence, 
\[ A \leq \alpha^2 | \ZT_q^{\alpha} g(\qx, \omega) |^2,\]
and we obtain as lower optimal frame bound the estimate
\[  A_{\text{opt}} =  \alpha^2 \essinf_{(\qx, \omega) \in Q_\alpha \times Q_{1/\alpha}} |\ZT_q^\alpha g(\qx, \omega)|^2. \]
The proof for the upper optimal frame bound is similar and, therefore, it will be omitted.
\end{proof}

Since an orthonormal basis is a special case of a frame, the next lemma follows immediately.
\begin{lem} For a given window $g \in L^2(\mathbb{R}^2, \mathbb{R})$ and $\alpha >0,$ the quaternionic Gabor system $\Gq(g, \alpha, \frac{1}{\alpha})$ is an orthonormal basis for $L^2(\mathbb{R}^2,\mathbb{H})$ if and only if $ |\ZT_q^{\alpha} g(\qx, \omega)|^2 = \alpha^{-2}$ almost everywhere in $Q_\alpha \times Q_{1/\alpha}.$
\end{lem}
\begin{proof} If $\Gq(g, \alpha, \frac{1}{\alpha})$ is an orthonormal basis for $L^2(\mathbb{R}^2,\mathbb{H})$, then
\[ \sum_{\qk, \qn \in \mathbb{Z}^2} |\langle f, M_{\frac{\qk}{\alpha}} T_{\alpha \qn} g \rangle|^2_2 = \|f\|^2_2 = \alpha^2 \|\ZT_q^{\alpha} f\|^2_{L^2(Q_\alpha \times Q_{1/\alpha})}.  \]
By Lemma \ref{lem:47} we have
\[  \alpha^4  \| \ZT_q^{\alpha} f \overline{\ZT_q^{\alpha} g}\|^2_{L^2(Q_\alpha \times Q_{1/\alpha})} = \alpha^2 \|\ZT_q^{\alpha} f\|^2_{L^2(Q_\alpha \times Q_{1/\alpha})} \]
which implies $| \ZT_q^{\alpha} g(\qx, \omega)|^2 = \alpha^{-2}$ for almost all $(\qx, \omega) \in  Q_{\alpha} \times Q_{1/\alpha}.$\\

We assume now that $|\ZT_q^{\alpha} g(\qx, \omega)|^2 = \alpha^{-2}$ for almost all $(\qx, \omega) \in Q_{\alpha} \times Q_{1/\alpha}.$ Then, by Lemmas \ref{lem:47} and \ref{lem:GqFrame}, $M_{\frac{\qk}{\alpha}} T_{\alpha \qn} g$ is a tight frame for $L^2(\mathbb{R}^2,\mathbb{H}).$ Indeed, $A_{\text{opt}} = B_{\text{opt}} =1.$ Finally, due to Theorem \ref{thm:41}, $\|M_{\frac{\qk}{\alpha}} T_{\alpha \qn} g\|_2 = \|g\|_2 = \alpha^2 \|\ZT_q^\alpha g\|_{L^2(Q_\alpha \times Q_{1/\alpha})} = 1$. This result combined with Lemma 5.1.6 (a) of \cite{Grchenig2001} implies that $\Gq(g, \alpha, \frac{1}{\alpha})$ is an orthonormal basis for $L^2(\mathbb{R}^2,\mathbb{H})$.
\end{proof}

The Gaussian function plays an important role in time-frequency analysis, since it minimizes the Heisenberg uncertainty principle. We have already seen that the two-dimensional Gaussian also minimizes the Heisenberg uncertainty principle in our two-sided quaternionic setting. Motivated by this, we establish the following theorem. 
\begin{thm}
The quaternionic Gabor system $\Gq(\exp(- \pi \qx^2), \alpha, \frac{1}{\alpha})$ is not a frame for $L^2(\mathbb{R}^2,\mathbb{H})$.
\end{thm}
\begin{proof}
In the first step we demonstrate that for any function $g \in W_0^\alpha$, the quaternionic Zak transform $\ZT_q^\alpha g$ is continuous. And in the second step we show that the Zak transform of $\exp(-\pi \qx^2) \in W_0^\alpha(\mathbb{R}^2)$ has at least one zero. Therefore, Lemma \ref{lem:GqFrame} can not hold, as $A_{\text{opt}} = 0$. \\
Step (1):\\
Given $\varepsilon > 0$, there exists a $N> 0$ such that\\
\[ \sum_{|\qk| > N} \| g T_{\alpha \qk} \mathcal{X}_{Q_{\alpha}} \|_\infty < \frac{\varepsilon}{4}. \]
Then, the main term $\sum_{|\qk| \leq N} \exp(2 \pi j \alpha k_2 \omega_2) g (\qx - \alpha \qk) \exp(2 \pi i \alpha k_1 \omega_1)$ is uniformly continuous on compact sets of $\mathbb{R}^4,$ and there exists a $\delta > 0$ such that
\[ \left| \sum_{|\qk| \leq N} \exp(2 \pi j \alpha k_2 \omega_2) g (\qx - \alpha\qk) \exp(2 \pi i \alpha k_1 \omega_1) - \sum_{|\qk| \leq N} \exp(2 \pi j \alpha k_2 \xi_2) g (\qy - \alpha \qk) \exp(2 \pi i \alpha k_1 \xi_1)\right| < \frac{\varepsilon}{2} \]
whenever $|\qx - \qy| + |\omega - \xi| < \delta.$ As a consequence, $|\ZT_q^{\alpha} g (\qx, \omega) - \ZT^{\alpha}_q g (\qy, \xi)| < \varepsilon$ and $\ZT^{\alpha}_q g$ is continuous.\\
Step (2):\\
We have for $\varphi(\qx) = \exp(- \pi x_1^2 - \pi x_2^2)$
\begin{align*}
\ZT^{\alpha}_q \varphi(\qx, \omega) &= \sum_{ \qm \in \mathbb{Z}^2} \exp(2 \pi j \alpha m_2 \omega_2) \exp(- \pi (\qx - \alpha \qm)^2) \exp(2 \pi i \alpha m_1 \omega_1)\\[1ex]
&=\exp(- \pi x_2^2) \sum_{\qm \in \mathbb{Z}^2} \exp(2 \pi j \alpha m_2 \omega_2 + 2 \pi \alpha x_2 m_2 - \pi \alpha^2 m_2^2)\\[1ex]
&\qquad \exp(- \pi \alpha^2 m_1^2 + 2 \pi \alpha m_1 x_1 + 2 \pi i \alpha m_1 \omega_1) \exp(- \pi x_1^2)
\end{align*}
\begin{align*}
&= \exp(- \pi x_2^2) \sum_{\qm \in \mathbb{Z}^2} \exp(2 \pi j \alpha m_2 (x_2 + j \omega_2) - \pi \alpha^2 m_2^2)\\[1ex]
&\qquad \exp(2 \pi i \alpha m_1(x_1 - i \omega_1) - \pi \alpha^2 m_1^2) \exp(- \pi x_1^2).
\end{align*}
We set now $\omega_1 = \omega_2 = \frac{1}{2 \alpha}$ and $x_1 = x_2 = \frac{\alpha}{2}$ and get for the inner sum
\begin{align*}
&\sum_{\qm \in \mathbb{Z}^2} \exp(m_2 \pi j - \pi \alpha^2 m_2 - \pi \alpha^2 m^2_2) \exp(m_1 \pi i - \pi \alpha^2 m_1 - \pi \alpha^2 m_1^2)\\[1ex]
&\sum_{\qm \in \mathbb{Z}^2} (-1)^{m_2} \exp(\pi \alpha^2 (-m_2 - m_2^2))\, (-1)^{m_1} \exp(\pi \alpha^2 (-m_1 - m_1^2)).
\end{align*}
To see that this sum is zero, we match $m_2 \geq 0$ with $- m_2 - 1$ and observe that they have opposite parity. Moreover, $(-m_2 -1)^2 + (- m_2 -1) = m_2^2 + m_2$. The same argument goes for the second part of the sum. Therefore, $\ZT_q^\alpha \varphi(\qx, \omega)$ is zero for $\omega_1 = \omega_2 = \frac{1}{2 \alpha}$ and $x_1 = x_2 = \frac{\alpha}{2}$ and our proof is finished. 
\end{proof}
The last result could be seen as Gabor's conjecture in the quaternionic case.

\section{Conclusions}
Using the two-sided quaternionic Fourier transform, we established the theory of the WQFT and quaternionic Gabor systems. Due to the non-commutativity we showed that the usual interplay between the translation and modulation operators holds in a more elaborated way. Thus, the underlying structure could be interesting for future research. By using the quaternionic Zak transform we could obtain conditions which ensure the quaternionic Gabor system to constitute a quaternionic Gabor frame. The two-dimensional Gaussian is of special interest as a window function, as it minimizes the uncertainty principle. Motivated by this fact, we were able to show that the Gabor's conjecture for the quaternionic case does not hold.    

\section{Acknowledgements}
The author gratefully acknowledges the many helpful suggestions of S. Bernstein, P. Cerejeiras and U. Kähler during the preparation of the paper. This work was supported by the ERASMUS program and Portuguese funds through the CIDMA - Center for Research and Development in Mathematics and Applications, and the Portuguese Foundation for Science and Technology (``FCT - Funda\c{c}\~ao para a Ci\^encia e a Tecnologia''), within project UID/MAT/ 0416/2013.

\bibliographystyle{plain}
\bibliography{References}
\end{document}